\documentclass[12pt]{article}

\usepackage{amsmath,amsthm,amssymb}
\usepackage[top=30truemm,bottom=30truemm,left=25truemm,right=25truemm]{geometry}
\usepackage[dvips]{graphicx,color,psfrag}
\usepackage{bm}

\theoremstyle{plain}
\newtheorem{definition}{Definition}[section]
\newtheorem{thm}[definition]{Theorem}

\newtheorem{lem}[definition]{Lemma}
\newtheorem{cor}[definition]{Corollary}

\newtheorem{rem}[definition]{Remark}

\title{Integer group determinants for ${\rm C}_{2}^{2} \rtimes {\rm C}_{4}$}
\author{Yuka Yamaguchi and Naoya Yamaguchi}
\date{\today}

\begin{document}

\maketitle

\begin{abstract}
For any positive integer $n$, 
let ${\rm C}_{n}$ be the cyclic group of order $n$. 
We determine all possible values of the integer group determinant of ${\rm C}_{2}^{2} \rtimes {\rm C}_{4}$. 
\end{abstract}

\section{Introduction}
For a finite group $G$, 
let $x_{g}$ be an indeterminate for each $g \in G$ and let $\mathbb{Z}[x_{g}]$ be the multivariate polynomial ring in $x_{g}$ over $\mathbb{Z}$. 
The group determinant $\Theta_{G}(x_{g})$ of $G$ was defined by Dedekind as follows: 
$$
\Theta_{G}(x_{g}) := \det{\left( x_{g h^{- 1}} \right)_{g, h \in G}} \in \mathbb{Z}[x_{g}]. 
$$
Frobenius \cite{Frobenius1968gruppen} gave the irreducible factorization of $\Theta_{G}(x_{g})$ over $\mathbb{C}$: 
{\it Let $G$ be a finite group and let $\widehat{G}$ be a complete set of representatives of the equivalence classes of irreducible representations of $G$ over $\mathbb{C}$. Then 
$$
\Theta_{G}(x_{g}) = \prod_{\varphi \in \widehat{G}} \det{\left( \sum_{g \in G} \varphi(g) x_{g} \right)^{\deg{\varphi}}}, 
$$
where $\deg{\varphi}$ is the degree of $\varphi$.}
This is called Frobenius' theorem. 

At the meeting of the American Mathematical Society in Hayward, California, in April 1977, 
Olga Taussky-Todd \cite{Olga} asked whether one could characterize the values of $\Theta_{G}(x_{g})$ when the all entries are integers \cite[p.1]{https://doi.org/10.48550/arxiv.2302.11688}. 
Let $S(G)$ denote the set of all achievable values: 
$$
S(G) := \left\{ \Theta_{G}(x_{g}) \mid x_{g} \in \mathbb{Z} \right\}. 
$$
For some groups, $S(G)$ was determined in \cite{MR3879399, MR2914452, MR624127, MR601702, MR550657, MR4363104, https://doi.org/10.48550/arxiv.2211.09930, https://doi.org/10.48550/arxiv.2302.11688, MR3998922, MR4056860, https://doi.org/10.48550/arxiv.2203.14420, Yamaguchi, https://doi.org/10.48550/arxiv.2211.14761, https://doi.org/10.48550/arxiv.2209.12446, https://doi.org/10.48550/arxiv.2211.01597}. 
As a result, for every group $G$ of order at most $15$, a complete description was obtained for $S(G)$. 
Also, for all five abelian groups of order $16$, and the non-abelian groups ${\rm D}_{16}, {\rm D}_{8} \times {\rm C}_{2}, {\rm Q}_{8} \times {\rm C}_{2}, {\rm Q}_{16}$, 
$S(G)$ have been determined. 
In this paper,  we determine $S \left( {\rm C}_{2}^{2} \rtimes {\rm C}_{4} \right)$, 
where 
$$
{\rm C}_{2}^{2} \rtimes {\rm C}_{4} = \langle g_{1}, g_{2}, g_{3} \mid g_{1}^{2} = g_{2}^{2} = g_{3}^{4} = e, \: g_{1} g_{2} = g_{2} g_{1}, \: g_{1} g_{3} = g_{3} g_{1}, \: g_{2} g_{3} = g_{3} g_{2} g_{1}  \rangle, 
$$
where $e$ is the unit element of G. 
Here, $e$ denotes the unit element. 
\begin{thm}\label{thm:1.1}
We have 
\begin{align*}
S({\rm C}_{2}^{2} \rtimes {\rm C}_{4}) = \left\{ 16 m + 1 \mid m \in \mathbb{Z} \right\} \cup 2^{14} \mathbb{Z}. 
\end{align*}
\end{thm}

\section{Preliminaries}
For any $\overline{r} \in {\rm C}_{4}$ with $r \in \{ 0, 1, 2, 3 \}$, 
we denote the variable $x_{\overline{r}}$ by $x_{r}$, 
and let 
$$
D_{4}(x_{0}, x_{1}, x_{2}, x_{3}) := \det{\left( x_{g h^{- 1}} \right)}_{g, h \in {\rm C}_{4}}. 
$$
For any $(\overline{r}, \overline{s}) \in {\rm C}_{4} \times {\rm C}_{2}$ with $r \in \{ 0, 1, 2, 3 \}$ and $s \in \{ 0, 1 \}$, 
we denote the variable $y_{(\overline{r}, \overline{s})}$ by $y_{j}$, where $j := r + 4 s$, 
and let 
$$
D_{4 \times 2}(y_{0}, y_{1}, \ldots, y_{7}) := \det{\left( y_{g h^{- 1}} \right)}_{g, h \in {\rm C}_{4} \times {\rm C}_{2}}. 
$$
From the $H = {\rm C}_{4}$ and $K = {\rm C}_{2}$ case of \cite[Theorem~1.1]{https://doi.org/10.48550/arxiv.2202.06952}, 
we have the following corollary. 

\begin{cor}\label{cor:2.1}
We have 
$$
D_{4 \times 2}(y_{0}, y_{1}, \ldots, y_{7}) = D_{4}(y_{0} + y_{4}, y_{1} + y_{5}, y_{2} + y_{6}, y_{3} + y_{7}) D_{4}(y_{0} - y_{4}, y_{1} - y_{5}, y_{2} - y_{6}, y_{3} - y_{7}). 
$$
\end{cor}

Let $G := {\rm C}_{2}^{2} \rtimes {\rm C}_{4}$. 
For any $g = g_{1}^{r} g_{2}^{s} g_{3}^{t} \in G$ with $r, s \in \{ 0, 1 \}$ and $t \in \{ 0, 1, 2, 3 \}$, 
we denote the variable $z_{g}$ by $z_{j}$, where $j := t + 4 s + 8 r$, 
and let 
$$
D_{G}(z_{0}, z_{1}, \ldots, z_{15}) := \det{(z_{g h^{- 1}})_{g, h \in G}}. 
$$
Let $[x] := \max{ \{ n \in \mathbb{Z} \mid n \leq x \} }$ and $\widehat{G} := \left\{ \varphi_{k} \mid 0 \leq k \leq 9 \right\}$, 
where 
\begin{align*}
\varphi_{k}(g_{1}) := 1, \quad \varphi_{k}(g_{2}) := (- 1)^{[ \frac{k}{4} ]}, \quad \varphi_{k}(g_{3}) := \sqrt{- 1}^{k} 
\end{align*}
for any $0 \leq k \leq 7$ and 
\begin{align*}
\varphi_{k}(g_{1}) := - \begin{pmatrix} 1 & 0 \\ 0 & 1 \end{pmatrix}, \quad
\varphi_{k}(g_{2}) := (- 1)^{k} \begin{pmatrix} 0 & 1 \\ 1 & 0 \end{pmatrix}, \quad
\varphi_{k}(g_{3}) := \sqrt{- 1}^{k + 1} \begin{pmatrix} 1 & 0 \\ 0 & - 1 \end{pmatrix}  
\end{align*}
for any $8 \leq k \leq 9$. 
Then, $\widehat{G}$ is a complete set of representatives of the equivalence classes of irreducible representations of $G$ over $\mathbb{C}$. 
Let $F$ be the eight variable polynomial defined by 
$$
F(w_{0}, w_{1}, \ldots, w_{7}) := f(w_{0} - w_{2}, w_{1} - w_{3}, w_{4} - w_{6}, w_{5} - w_{7}) f(w_{5} + w_{7}, w_{0} + w_{2}, w_{1} + w_{3}, w_{4} + w_{6}),
$$
where $f(x, y, z, w) := x^{2} + y^{2} - z^{2} - w^{2}$. 
Then we have the following. 

\begin{lem}\label{lem:2.2}
We have 
$$
D_{G}(z_{0}, z_{1}, \ldots, z_{15}) = D_{4 \times 2}(z_{0} + z_{8}, z_{1} + z_{9}, \ldots, z_{7} + z_{15}) F(z_{0} - z_{8}, z_{1} - z_{9}, \ldots, z_{7} - z_{15})^{2}. 
$$
\end{lem}
\begin{proof}
From Frobenius' theorem, we have 
\begin{align*}
D_{G}(z_{0}, z_{1}, \ldots, z_{15}) = \prod_{\varphi \in \widehat{G}} \det{\left( \sum_{g \in G} \varphi(g) z_{g} \right)}^{\deg{\varphi}} = \prod_{k = 0}^{9} \det{M_{k}}^{\deg{\varphi_{k}}}, 
\end{align*}
where $M_{k} := \sum_{0 \leq r \leq 1} \sum_{0 \leq s \leq 1} \sum_{0 \leq t \leq 3} \varphi_{k}(g_{1}^{r} g_{2}^{s} g_{3}^{t}) z_{t + 4 s + 8 r}$. 
Note that $\{ \chi_{k} \mid 0 \leq k \leq 7 \}$ is the dual group of ${\rm C}_{4} \times {\rm C}_{2}$, where 
$\chi_{k}( (\overline{1}, \overline{0}) ) := \sqrt{- 1}^{k}$ and $\chi_{k}( (\overline{0}, \overline{1}) ) := (- 1)^{[\frac{k}{4}]}$ for $0 \leq k \leq 7$. 
Then we have 
\begin{align*}
\prod_{k = 0}^{7} \det{M_{k}}^{\deg{\varphi_{k}}} 
&= \prod_{k = 0}^{7} M_{k} \\ 
&= \prod_{k = 0}^{7} \sum_{s = 0}^{1} \sum_{t = 0}^{3} (- 1)^{[\frac{k}{4}] s} \sqrt{- 1}^{k t} (z_{t + 4 s} + z_{t + 4 s + 8}) \\
&= \prod_{k = 0}^{7} \sum_{s = 0}^{1} \sum_{t = 0}^{3} \chi_{k}( (\overline{t}, \overline{s}) ) (z_{t + 4 s} + z_{t + 4 s + 8}) \\ 
&= D_{4 \times 2}(z_{0} + z_{8}, z_{1} + z_{9}, \ldots, z_{7} + z_{15}). 
\end{align*}
Also, for any $8 \leq k \leq 9$, we have 
\begin{align*}
M_{k} 
&= \sum_{s = 0}^{1} \sum_{t = 0}^{3} (- 1)^{k s} \sqrt{- 1}^{(k + 1) t} (z_{t + 4 s} - z_{t + 4 s + 8}) 
\begin{pmatrix} 0 & 1 \\ 1 & 0 \end{pmatrix}^{s} \begin{pmatrix} 1 & 0 \\ 0 & - 1 \end{pmatrix}^{t} \\ 
&= \sum_{t = 0}^{3} \left\{ \sqrt{- 1}^{(k + 1) t} (z_{t} - z_{t + 8}) \begin{pmatrix} 1 & 0 \\ 0 & (- 1)^{t} \end{pmatrix} + (- 1)^{k} \sqrt{- 1}^{(k + 1) t} (z_{t + 4} - z_{t + 12}) \begin{pmatrix} 0 & (- 1)^{t} \\ 1 & 0 \end{pmatrix} \right\} \\ 
&= 
\begin{pmatrix} 
\sum_{t = 0}^{3} \sqrt{- 1}^{(k + 1) t} \tilde{z}_{t} & \sum_{t = 0}^{3} (- 1)^{k + t} \sqrt{- 1}^{(k + 1) t} \tilde{z}_{t + 4} \\ 
\sum_{t = 0}^{3} (- 1)^{k} \sqrt{- 1}^{(k + 1) t} \tilde{z}_{t + 4} & \sum_{t = 0}^{3} (- 1)^{t} \sqrt{- 1}^{(k + 1) t} \tilde{z}_{t} 
\end{pmatrix}, 
\end{align*}
where $\tilde{z}_{t} := z_{t} - z_{t + 8}$ for $0 \leq t \leq 7$. 
Therefore, 
\begin{align*}
\det{M_{k}} 
&= \sum_{t = 0}^{3} \sqrt{- 1}^{(k + 1) t} \tilde{z}_{t} \sum_{u = 0}^{3} (- 1)^{u} \sqrt{- 1}^{(k + 1) u} \tilde{z}_{u} \\ 
&\qquad - \sum_{t = 0}^{3} \sqrt{- 1}^{(k + 1) t} \tilde{z}_{t + 4} \sum_{u = 0}^{3} (- 1)^{u} \sqrt{- 1}^{(k + 1) u} \tilde{z}_{u + 4} \\ 
&= \left\{ \tilde{z}_{0} + (- 1)^{k + 1} \tilde{z}_{2} \right\}^{2} - (- 1)^{k + 1} \left\{ \tilde{z}_{1} + (- 1)^{k + 1} \tilde{z}_{3} \right\}^{2} \\
&\qquad - \left\{ \tilde{z}_{4} + (- 1)^{k + 1} \tilde{z}_{6} \right\}^{2} + (- 1)^{k + 1} \left\{ \tilde{z}_{5} + (- 1)^{k + 1} \tilde{z}_{7} \right\}^{2}. 
\end{align*}
That is, 
\begin{align*}
\det{M_{8}} = f(\tilde{z}_{0} - \tilde{z}_{2}, \tilde{z}_{1} - \tilde{z}_{3}, \tilde{z}_{4} - \tilde{z}_{6}, \tilde{z}_{5} - \tilde{z}_{7}), \quad 
\det{M_{9}} = f(\tilde{z}_{5} + \tilde{z}_{7}, \tilde{z}_{0} + \tilde{z}_{2}, \tilde{z}_{1} + \tilde{z}_{3}, \tilde{z}_{4} + \tilde{z}_{6}). 
\end{align*} 
This completes the proof. 
\end{proof}

Throughout this paper, 
we assume that $a_{0}, a_{1}, \ldots, a_{15} \in \mathbb{Z}$ and let 
\begin{align*}
&&b_{i} &:= (a_{i} + a_{i + 8}) + (a_{i + 4} + a_{i + 12}) && (0 \leq i \leq 3), \\
&&c_{i} &:= (a_{i} + a_{i + 8}) - (a_{i + 4} + a_{i + 12}) && (0 \leq i \leq 3), \\
&&d_{i} &:= a_{i} - a_{i + 8} && (0 \leq i \leq 7). &
\end{align*}
Also, let 
\begin{align*}
\bm{a} := (a_{0}, a_{1}, \ldots, a_{15}), \quad 
\bm{b} := (b_{0}, b_{1}, b_{2}, b_{3}), \quad 
\bm{c} := (c_{0}, c_{1}, c_{2}, c_{3}), \quad 
\bm{d} := (d_{0}, d_{1}, \ldots, d_{7}). 
\end{align*}
Then, from Corollary~$\ref{cor:2.1}$ and Lemma~$\ref{lem:2.2}$, 
we have 
\begin{align*}
D_{G}(\bm{a}) 
&= D_{4 \times 2}(a_{0} + a_{8}, a_{1} + a_{9}, \ldots, a_{7} + a_{15}) F(a_{0} - a_{8}, a_{1} - a_{9}, \ldots, a_{7} - a_{15})^{2} \\ 
&= D_{4}(\bm{b}) D_{4}(\bm{c}) F(\bm{d})^{2}. 
\end{align*}

\begin{rem}\label{rem:2.3}
For any $0 \leq i \leq 3$, we have $b_{i} \equiv c_{i} \equiv d_{i} + d_{i + 4} \pmod{2}$. 
\end{rem}

\begin{lem}\label{lem:2.4}
The following hold: 
\begin{enumerate}
\item[$(1)$] $D_{4}(\bm{b}) \equiv b_{0} + b_{1} + b_{2} + b_{3} \pmod{2}$; 
\item[$(2)$] $D_{4}(\bm{c}) \equiv c_{0} + c_{1} + c_{2} + c_{3} \pmod{2}$; 
\item[$(3)$] $F(\bm{d}) \equiv d_{0} + d_{1} + \cdots + d_{7} \pmod{2}$. 
\end{enumerate}
\end{lem}
\begin{proof}
We obtain (1) and (2) from the following: for any $x_{0}, x_{1}, x_{2}, x_{3} \in \mathbb{Z}$, 
\begin{align*}
D_{4}(x_{0}, x_{1}, x_{2}, x_{3}) 
&= \left\{ (x_{0} + x_{2})^{2} - (x_{1} + x_{3})^{2} \right\} \left\{ (x_{0} - x_{2})^{2} + (x_{1} - x_{3})^{2} \right\} \\ 
&\equiv (x_{0}^{2} + x_{2}^{2} + x_{1}^{2} + x_{3}^{2})^{2} \\ 
&\equiv x_{0} + x_{1} + x_{2} + x_{3} \pmod{2}. 
\end{align*}
We prove (3). 
Since $f(x, y, z, w) = x^{2} + y^{2} - z^{2} - w^{2} \equiv x + y + z + w \pmod{2}$ for any $x, y, z, w \in \mathbb{Z}$, 
we have 
\begin{align*}
F(\bm{d}) 
&= f(d_{0} - d_{2}, d_{1} - d_{3}, d_{4} - d_{6}, d_{5} - d_{7}) f(d_{5} + d_{7}, d_{0} + d_{2}, d_{1} + d_{3}, d_{4} + d_{6}) \\ 
&\equiv (d_{0} + d_{1} + \cdots + d_{7})^{2} \\ 
&\equiv d_{0} + d_{1} + \cdots + d_{7} \pmod{2}. 
\end{align*}
\end{proof}

From Remark~$\ref{rem:2.3}$ and Lemma~$\ref{lem:2.4}$, 
we have the following lemma. 

\begin{lem}\label{lem:2.5}
We have $D_{G}(\bm{a}) \equiv D_{4}(\bm{b}) \equiv D_{4}(\bm{c}) \equiv F(\bm{d}) \pmod{2}$. 
\end{lem}

\section{Impossible integers}\label{section3}
In this section, 
we consider impossible integers. 
Let $\mathbb{Z}_{\rm odd}$ be the set of all odd numbers. 

\begin{lem}\label{lem:3.1}
We have $S({\rm C}_{2}^{2} \rtimes {\rm C}_{4}) \cap \mathbb{Z}_{\rm odd} \subset \{ 16 m + 1 \mid m \in \mathbb{Z} \}$. 
\end{lem}

\begin{lem}\label{lem:3.2}
We have $S({\rm C}_{2}^{2} \rtimes {\rm C}_{4}) \cap 2 \mathbb{Z} \subset 2^{14} \mathbb{Z}$. 
\end{lem}

Let $m_{0} := f(d_{0} - d_{2}, d_{1} - d_{3}, d_{4} - d_{6}, d_{5} - d_{7})$ and $m_{1} := f(d_{5} + d_{7}, d_{0} + d_{2}, d_{1} + d_{3}, d_{4} + d_{6})$. 
To prove Lemma~$\ref{lem:3.1}$, we use the following two lemmas. 

\begin{lem}\label{lem:3.3}
Suppose that $b_{0} + b_{2} \not\equiv b_{1} + b_{3} \pmod{2}$. 
Then we have 
$$
D_{4}(\bm{b}) D_{4}(\bm{c}) \equiv 1 - 8 (d_{0} d_{2} + d_{4} d_{6} + d_{1} d_{3} + d_{5} d_{7}) \pmod{16}. 
$$
\end{lem}
\begin{proof}
From \cite[Lemmas~2.10 (1) and 4.3]{https://doi.org/10.48550/arxiv.2211.01597}, 
the lemma is obtained. 
\end{proof}

\begin{lem}\label{lem:3.4}
Suppose that $b_{0} + b_{2} \not\equiv b_{1} + b_{3} \pmod{2}$. 
Then we have 
$$
F(\bm{d})^{2} \equiv 1 - 8 (d_{0} d_{2} + d_{4} d_{6} + d_{1} d_{3} + d_{5} d_{7}) \pmod{16}. 
$$
\end{lem}
\begin{proof}
From $b_{0} + b_{2} \not\equiv b_{1} + b_{3} \pmod{2}$, 
we have 
$$
d_{0} + d_{2} + d_{4} + d_{6} \not\equiv d_{1} + d_{3} + d_{5} + d_{7} \pmod{2}. 
$$
From this, $m_{0} \equiv m_{1} \equiv d_{0} + d_{1} + \cdots + d_{7} \equiv 1 \pmod{2}$. 
Therefore, $m_{0}^{2} \equiv m_{1}^{2} \equiv 1 \pmod{8}$. 
Also, from 
\begin{align*}
m_{0} + m_{1} 
&= 2 (d_{0} + d_{2})^{2} - 2 (d_{4} + d_{6})^{2} + 4 (- d_{0} d_{2} + d_{4} d_{6} - d_{1} d_{3} + d_{5} d_{7}) \\ 
&= 2 (d_{0} + d_{2} + d_{4} + d_{6}) (d_{0} + d_{2} - d_{4} - d_{6}) + 4 (- d_{0} d_{2} + d_{4} d_{6} - d_{1} d_{3} + d_{5} d_{7}), \\ 
m_{0} - m_{1} 
&= 2 (d_{1} + d_{3})^{2} - 2 (d_{5} + d_{7})^{2} + 4 (- d_{0} d_{2} + d_{4} d_{6} - d_{1} d_{3} + d_{5} d_{7}) \\ 
&= 2 (d_{1} + d_{3} + d_{5} + d_{7}) (d_{1} + d_{3} - d_{5} - d_{7}) + 4 (- d_{0} d_{2} + d_{4} d_{6} - d_{1} d_{3} + d_{5} d_{7}), 
\end{align*}
we have 
\begin{align*}
m_{0}^{2} - m_{1}^{2} 
&= (m_{0} + m_{1}) (m_{0} - m_{1}) \\ 
&\equiv 8 \left\{ (d_{0} + d_{2} + d_{4} + d_{6}) (d_{0} + d_{2} - d_{4} - d_{6}) + (d_{1} + d_{3} + d_{5} + d_{7}) (d_{1} + d_{3} - d_{5} - d_{7}) \right\} \\ 
&\qquad \times (- d_{0} d_{2} + d_{4} d_{6} - d_{1} d_{3} + d_{5} d_{7}) \\ 
&\equiv 8 (d_{0} d_{2} + d_{4} d_{6} + d_{1} d_{3} + d_{5} d_{7}) \pmod{16}. 
\end{align*}
From the above, 
we have 
\begin{align*}
F(\bm{d})^{2} 
&= m_{0}^{2} m_{1}^{2} \\ 
&\equiv m_{0}^{2} \left\{ m_{0}^{2} - 8 (d_{0} d_{2} + d_{4} d_{6} + d_{1} d_{3} + d_{5} d_{7}) \right\} \\ 
&\equiv 1 - 8 (d_{0} d_{2} + d_{4} d_{6} + d_{1} d_{3} + d_{5} d_{7}) \pmod{16}. 
\end{align*}
\end{proof}

\begin{proof}[Proof of Lemma~$\ref{lem:3.1}$]
Suppose that $D_{G}(\bm{a}) = D_{4}(\bm{b}) D_{4}(\bm{c}) F(\bm{d})^{2} \in \mathbb{Z}_{\rm odd}$. 
Then, $D_{4}(\bm{b}) \in \mathbb{Z}_{\rm odd}$. 
From this and Lemma~$\ref{lem:2.4}$ (1), we have $b_{0} + b_{2} \not\equiv b_{1} + b_{3} \pmod{2}$. 
Therefore, from Lemmas~$\ref{lem:3.3}$ and $\ref{lem:3.4}$, we have 
\begin{align*}
D_{G}(\bm{a}) = D_{4}(\bm{b}) D_{4}(\bm{c}) F(\bm{d})^{2} \equiv \left\{ 1 - 8 (d_{0} d_{2} + d_{4} d_{6} + d_{1} d_{3} + d_{5} d_{7}) \right\}^{2} \equiv 1 \pmod{16}. 
\end{align*}
\end{proof}

To prove Lemma~$\ref{lem:3.2}$, we use the following lemma. 

\begin{lem}\label{lem:3.5}
Suppose that $b_{0} + b_{2} \equiv b_{1} + b_{3} \pmod{2}$. 
Then we have $F(\bm{d}) \equiv 0 \pmod{8}$. 
\end{lem}
\begin{proof}
First, we consider the case $b_{0} + b_{2} \equiv b_{1} + b_{3} \equiv 0 \pmod{2}$. 
From 
$$
d_{0} + d_{2} + d_{4} + d_{6} \equiv d_{1} + d_{3} + d_{5} + d_{7} \equiv 0 \pmod{2}, 
$$
we have 
$$
m_{0} = (d_{0} - d_{2} + d_{4} - d_{6}) (d_{0} - d_{2} - d_{4} + d_{6}) + (d_{1} - d_{3} + d_{5} - d_{7}) (d_{1} - d_{3} - d_{5} + d_{7}) \equiv 0 \pmod{4}. 
$$
In the same way, 
we can obtain $m_{1} \equiv 0 \pmod{4}$. 
Thus we have $F(\bm{d}) = m_{0} m_{1} \equiv 0 \pmod{16}$. 
Next, we consider the case $b_{0} + b_{2} \equiv b_{1} + b_{3} \equiv 1 \pmod{2}$. 
From 
$$
d_{0} + d_{2} + d_{4} + d_{6} \equiv d_{1} + d_{3} + d_{5} + d_{7} \equiv 1 \pmod{2}, 
$$
we have $m_{0} \equiv m_{1} \equiv 0 \pmod{2}$ and 
\begin{align*}
m_{0} - m_{1} 
&= 2 (d_{1} + d_{3})^{2} - 2 (d_{5} + d_{7})^{2} + 4 (- d_{0} d_{2} + d_{4} d_{6} - d_{1} d_{3} + d_{5} d_{7}) \\ 
&\equiv 2 (d_{1}^{2} + d_{3}^{2} - d_{5}^{2} - d_{7}^{2}) \\ 
&\equiv 2 (d_{1} + d_{3} + d_{5} + d_{7}) \\ 
&\equiv 2 \pmod{4}. 
\end{align*}
Thus we have $F(\bm{d}) = m_{0} m_{1} \equiv 0 \pmod{8}$. 
\end{proof}

\begin{proof}[Proof of Lemma~$\ref{lem:3.2}$]
Suppose that $D_{G}(\bm{a}) = D_{4}(\bm{b}) D_{4}(\bm{c}) F(\bm{d})^{2} \in 2 \mathbb{Z}$. 
Then from Lemma~$\ref{lem:2.5}$ and $S({\rm C}_{4}) = \mathbb{Z}_{\rm odd} \cup 2^{4} \mathbb{Z}$ (\cite[Theorem~1.1]{MR2914452}), 
we have $D_{4}(\bm{b}), D_{4}(\bm{c}) \in 2^{4} \mathbb{Z}$. 
Also, from Lemma~$\ref{lem:2.4}$ (1), 
we have $b_{0} + b_{1} + b_{2} + b_{3} \equiv D_{4}(\bm{b}) \equiv 0 \pmod{2}$. 
Thus, from Lemma~$\ref{lem:3.5}$, 
we have $F(\bm{d}) \in 2^{3} \mathbb{Z}$. 
From the above, we have 
$D_{G}(\bm{a}) = D_{4}(\bm{b}) D_{4}(\bm{c}) F(\bm{d})^{2} \in 2^{14} \mathbb{Z}$. 
\end{proof}

\section{Possible integers}
In this section, 
we determine all possible integers. 
Lemmas~$\ref{lem:3.1}$ and $\ref{lem:3.2}$ imply that $S \left( {\rm C}_{2}^{2} \rtimes {\rm C}_{4} \right)$ does not include every integer that is not mentioned in the following Lemma~$\ref{lem:4.1}$. 

\begin{lem}\label{lem:4.1}
For any $m \in \mathbb{Z}$, 
the following are elements of $S({\rm C}_{2}^{2} \rtimes {\rm C}_{4})$: 
\begin{enumerate}
\item[$(1)$] $16 m + 1$; 
\item[$(2)$] $2^{14} (4 m + 1)$; 
\item[$(3)$] $- 2^{14} (4 m + 1)$; 
\item[$(4)$] $2^{15} (2 m + 1)$; 
\item[$(5)$] $2^{16} m$. 
\end{enumerate}
\end{lem}
\begin{proof}
We obtain (1) from 
\begin{align*}
D_{G}(m + 1, m, m, \ldots, m) 
&= D_{4 \times 2}(2 m + 1, 2 m, 2 m, \ldots, 2 m) F(1, 0, 0, \ldots, 0)^{2} \\ 
&= D_{4}(4 m + 1, 4 m, 4 m, 4 m) D_{4}(1, 0, 0, 0) f(1, 0, 0, 0)^{2} f(0, 1, 0, 0)^{2} \\ 
&= \left\{ (8 m + 1)^{2} - (8 m)^{2} \right\} \cdot 1 \cdot 1^{2} \cdot 1^{2} \\ 
&= 16 m + 1. 
\end{align*}
We obtain (2) from 
\begin{align*}
&D_{G}(m + 1, m, m + 1, m + 1, m, m + 1, m, m, m + 1, m, m, m, m, m - 1, m, m) \\ 
&\quad = D_{4 \times 2}(2 m + 2, 2 m, 2 m + 1, 2 m + 1, 2 m, 2 m, 2 m, 2 m) F(0, 0, 1, 1, 0, 2, 0, 0)^{2} \\ 
&\quad = D_{4}(4 m + 2, 4 m, 4 m + 1, 4 m + 1) D_{4}(2, 0, 1, 1) f(- 1, - 1, 0, 2)^{2} f(2, 1, 1, 0)^{2} \\ 
&\quad = 2 \left\{ (8 m + 3)^{2} - (8 m + 1)^{2} \right\} \cdot 16 \cdot (- 2)^{2} \cdot 4^{2}\\ 
&\quad = 2^{14} (4 m + 1). 
\end{align*}
We obtain (3) from 
\begin{align*}
&D_{G}(m + 1, m + 1, m + 1, m, m, m + 1, m, m, m, m, m, m + 1, m - 1, m , m, m) \\ 
&\quad = D_{4 \times 2}(2 m + 1, 2 m + 1, 2 m + 1, 2 m + 1, 2 m - 1, 2 m + 1, 2 m, 2 m) F(1, 1, 1, - 1, 1, 1, 0, 0)^{2} \\ 
&\quad = D_{4}(4 m, 4 m + 2, 4 m + 1, 4 m + 1) D_{4}(2, 0, 1, 1) f(0, 2, 1, 1)^{2} f(1, 2, 0, 1)^{2} \\ 
&\quad = 2 \left\{ (8 m + 1)^{2} - (8 m + 3)^{2} \right\} \cdot 16 \cdot 2^{2} \cdot 4^{2} \\ 
&\quad = - 2^{14} (4 m + 1). 
\end{align*}
We obtain (4) from 
\begin{align*}
&D_{G}(m + 1, m + 1, m + 1, m + 1, m, m + 1, m + 1, m, m + 1, m, m + 1, m, m, m , m, m) \\ 
&\quad = D_{4 \times 2}(2 m + 2, 2 m + 1, 2 m + 2, 2 m + 1, 2 m, 2 m + 1, 2 m + 1, 2 m) F(0, 1, 0, 1, 0, 1, 1, 0)^{2} \\ 
&\quad = D_{4}(4 m + 2, 4 m + 2, 4 m + 3, 4 m + 1) D_{4}(2, 0, 1, 1) f(0, 0, - 1, 1)^{2} f(1, 0, 2, 1)^{2} \\ 
&\quad = 2 \left\{ (8 m + 5)^{2} - (8 m + 3)^{2} \right\} \cdot 16 \cdot (- 2)^{2} \cdot (- 4)^{2} \\ 
&\quad = 2^{15} (2 m + 1). 
\end{align*}
We obtain (5) from 
\begin{align*}
&D_{G}(m + 1, m, m + 1, m, m, m, m, m, m, m, m, m, m - 1, m , m, m - 1) \\ 
&\quad = D_{4 \times 2}(2 m + 1, 2 m, 2 m + 1, 2 m, 2 m - 1, 2 m, 2 m, 2 m - 1) F(1, 0, 1, 0, 1, 0, 0, 1)^{2} \\ 
&\quad = D_{4}(4 m, 4 m, 4 m + 1, 4 m - 1) D_{4}(2, 0, 1, 1) f(0, 0, 1, - 1)^{2} f(1, 2, 0, 1)^{2} \\ 
&\quad = 2 \left\{ (8 m + 1)^{2} - (8 m - 1)^{2} \right\} \cdot 16 \cdot ( - 2)^{2} \cdot 4^{2} \\ 
&\quad = 2^{16} m. 
\end{align*}
\end{proof}

From Lemmas~$\ref{lem:3.1}$, $\ref{lem:3.2}$ and $\ref{lem:4.1}$, 
Theorem~$\ref{thm:1.1}$ is proved.

\clearpage

\bibliography{reference}

\begin{thebibliography}{10}

\bibitem{MR3879399}
Ton Boerkoel and Christopher Pinner.
\newblock Minimal group determinants and the {L}ind-{L}ehmer problem for
  dihedral groups.
\newblock {\em Acta Arith.}, 186(4):377--395, 2018.

\bibitem{Frobenius1968gruppen}
Ferdinand~Georg Frobenius.
\newblock \"{U}ber die {P}rimfactoren der {G}ruppendeterminante.
\newblock {\em Sitzungsberichte der K\"{o}niglich Preu{\ss}ischen Akademie der
  Wissenschaften zu Berlin}, pages 1343--1382, 1896.
\newblock Reprinted in {\it Gesammelte Abhandlungen, Band III}. Springer-Verlag
  Berlin Heidelberg, New York, 1968, pages 38--77.

\bibitem{MR2914452}
Norbert Kaiblinger.
\newblock Progress on {O}lga {T}aussky-{T}odd's circulant problem.
\newblock {\em Ramanujan J.}, 28(1):45--60, 2012.

\bibitem{MR624127}
H.~Turner Laquer.
\newblock Values of circulants with integer entries.
\newblock In {\em A collection of manuscripts related to the {F}ibonacci
  sequence}, pages 212--217. Fibonacci Assoc., Santa Clara, Calif., 1980.

\bibitem{MR601702}
Morris Newman.
\newblock Determinants of circulants of prime power order.
\newblock {\em Linear and Multilinear Algebra}, 9(3):187--191, 1980.

\bibitem{MR550657}
Morris Newman.
\newblock On a problem suggested by {O}lga {T}aussky-{T}odd.
\newblock {\em Illinois J. Math.}, 24(1):156--158, 1980.

\bibitem{MR4363104}
Bishnu Paudel and Chris Pinner.
\newblock Integer circulant determinants of order 15.
\newblock {\em Integers}, 22:Paper No. A4, 21, 2022.

\bibitem{https://doi.org/10.48550/arxiv.2211.09930}
Bishnu Paudel and Christopher Pinner.
\newblock The group determinants for $\mathbb{Z}_n \times {H}$, 2022.
\newblock arXiv:2211.09930v3 [math.NT].

\bibitem{https://doi.org/10.48550/arxiv.2302.11688}
Bishnu Paudel and Christopher Pinner.
\newblock The integer group determinants for ${Q}_{16}$, 2023.
\newblock arXiv:2302.11688v1 [math.NT].

\bibitem{MR3998922}
Christopher Pinner.
\newblock The integer group determinants for the symmetric group of degree
  four.
\newblock {\em Rocky Mountain J. Math.}, 49(4):1293--1305, 2019.

\bibitem{MR4056860}
Christopher Pinner and Christopher Smyth.
\newblock Integer group determinants for small groups.
\newblock {\em Ramanujan J.}, 51(2):421--453, 2020.

\bibitem{Olga}
Olga~Taussky Todd.
\newblock Integral group matrices.
\newblock In {\em Notices of the American Mathematical Society}, volume~24,
  United States of America, April 1977. American Mathematical Society.
\newblock Abstract no. 746-A15, 746th Meetting, Hayward, California, Apr.
  22--23, 1977.

\bibitem{https://doi.org/10.48550/arxiv.2203.14420}
Naoya Yamaguchi and Yuka Yamaguchi.
\newblock Generalized {D}edekind's theorem and its application to integer group
  determinants, 2022.
\newblock arXiv:2203.14420v2 [math.RT].

\bibitem{https://doi.org/10.48550/arxiv.2202.06952}
Naoya Yamaguchi and Yuka Yamaguchi.
\newblock Remark on {L}aquer's theorem for circulant determinants.
\newblock {\em International Journal of Group Theory}, 12(4):265--269, 2023.

\bibitem{Yamaguchi}
Yuka Yamaguchi and Naoya Yamaguchi.
\newblock Integer circulant determinants of order 16.
\newblock {\em The Ramanujan Journal}, 2022.

\bibitem{https://doi.org/10.48550/arxiv.2211.14761}
Yuka Yamaguchi and Naoya Yamaguchi.
\newblock Integer group determinants for abelian groups of order 16, 2022.
\newblock arXiv:2211.14761v1 [math.NT].

\bibitem{https://doi.org/10.48550/arxiv.2209.12446}
Yuka Yamaguchi and Naoya Yamaguchi.
\newblock Integer group determinants for ${{\rm C}}_{2}^{4}$, 2022.
\newblock arXiv:2209.12446v3 [math.NT].

\bibitem{https://doi.org/10.48550/arxiv.2211.01597}
Yuka Yamaguchi and Naoya Yamaguchi.
\newblock Integer group determinants for ${{\rm C}}_{4}^{2}$, 2022.
\newblock arXiv:2211.01597v2 [math.NT].

\end{thebibliography}
\bibliographystyle{plain}

\medskip
\begin{flushleft}
Faculty of Education, 
University of Miyazaki, 
1-1 Gakuen Kibanadai-nishi, 
Miyazaki 889-2192, 
Japan \\ 
{\it Email address}, Yuka Yamaguchi: y-yamaguchi@cc.miyazaki-u.ac.jp \\ 
{\it Email address}, Naoya Yamaguchi: n-yamaguchi@cc.miyazaki-u.ac.jp 
\end{flushleft}

\end{document}